\documentclass[12pt,a4paper,twoside,reqno]{amsart}

\usepackage{tikz}
\usepackage{amsmath, verbatim}
\usepackage{amssymb,amsfonts,mathrsfs}
\usepackage[colorlinks=true,linkcolor=blue,citecolor=blue]{hyperref}
\usepackage[driver=pdftex,margin=3cm,heightrounded=true,centering]{geometry}

\usepackage{mllocal} 

\DeclareMathOperator{\ck}{\mathscr{C}^{2k}_{\textup{ie}}(M, \mathcal{D})}
\DeclareMathOperator{\ckk}{\mathscr{C}^{2k+2}_{\textup{ie}}(M, \mathcal{D})}

\numberwithin{equation}{section}

\begin{document}


\title[The biharmonic heat operator on edge manifolds]
{The biharmonic heat operator on edge manifolds and non-linear fourth order equations}

\author{Boris Vertman}
\address{Mathematisches Institut,
Universit\"at Bonn,
53115 Bonn,
Germany}
\email{vertman@math.uni-bonn.de}
\urladdr{www.math.uni-bonn.de/people/vertman}

\date{\today}

\begin{abstract}
{We construct the biharmonic heat kernel for 
a suitable self-adjoint extension of the bi-Laplacian on
a manifold with incomplete edge singularities. We employ a microlocal 
description of the biharmonic heat kernel to establish mapping properties of the corresponding biharmonic heat operator
on certain Banach spaces. This yields short time existence for a class of semi-linear equations of fourth order,
including for example the Cahn-Hilliard equation. We also obtain asymptotics of the solutions near the edge singularity.}
\end{abstract}

\maketitle

\tableofcontents

\section{Introduction}\label{intro}

In this paper we provide a microlocal construction of the biharmonic heat kernel for a
self-adjoint extension of the Laplace operator on a manifold with incomplete edge singularities.
Such manifolds include spaces with isolated conical singularities, more precisely
open manifolds $(M,g)$ with a decomposition $M=K\cup_N \mathscr{U}$, where 
$K$ is a compact manifold with boundary $N$, $(N,g^N)$ is a closed Riemannian manifold, 
$\mathscr{U}=(0,1]\times N$ and 
$$
g\restriction \mathscr{U} = dx^2 \oplus x^2 g^N, \ x \in (0,1].
$$
While the heat kernel for the Laplace operator on edge manifolds has been studied extensively
before, compare for example joint work with Mazzeo \cite{MazVer:ATM}, Bahuaud \cite{BahVer:YFO}, 
Bahuaud and Dryden \cite{BDV:THE}, as well as \cite{Ver-Mooers}; the present work seems
to be the first step towards a microlocal analysis of the bi-Laplacian, its heat kernel and 
associated non-linear partial differential equations 
of fourth order on edge spaces. In the non-singular setting cf. Lamm \cite{Lam01}.
\medskip

Our construction of the biharmonic heat kernel yields a precise understanding of its 
asymptotic properties, which in turn allows to study the mapping properties 
of the corresponding biharmonic heat operator. In the present paper we concentrate on
the mapping properties with respect to certain Banach spaces that yield 
short time existence for some semi-linear equations of fourth order, in analogy to 
\cite{JefLoy:RSH} and \cite{BDV:THE}. \medskip

We point out that several other interesting open questions are delegated to future research and not touched upon here.
These include mapping properties of the biharmonic heat kernel with respect to certain
H\"older spaces and applications to short time existence of quasi-linear equations, as in 
\cite{BahVer:YFO}. Further questions are concerned with elliptic boundary value problems
for the bi-Laplacian and the corresponding biharmonic heat trace asymptotics, as in \cite{Ver-Mooers}.
\medskip

Our interest in non-linear fourth order equations of parabolic type stems from recent 
results on e.g. the Cahn-Hilliard equation on spaces with isolated conical singularities 
by Roidos and Schrohe \cite{RoSc12}, as well as results on higher order geometric flows on compact manifolds,
generated by powers of the Laplacian applied to the Ricci tensor or by the ambient obstruction tensor, introduced
by Fefferman and Graham \cite{FG85}, see Bahuaud and Helliwell \cite{BahDyl}.
\medskip

The Cahn-Hilliard equation was proposed by Cahn and Hilliard in \cite{Cah61}, \cite{CaHi58} 
as a simple model of the phase separation process, where at a fixed temperature the two components 
of a binary fluid spontaneously separate and form domains that are pure in each component.
Let $\Delta$ denote the Laplace Beltrami operator. Then the Cahn-Hilliard equation 
may be stated in the following form
\begin{equation*}
\partial_t u + \Delta^2 u + \Delta (u-u^3)) = 0,\ u(0)=u_0.
\end{equation*}

Global existence for solutions to the Cahn-Hilliard equation has been established by 
Elliott and Songmu \cite{ElSo86}, and Caffarelli and Muler \cite{CaMu95}. In the setup of singular manifolds however, 
there is still a question of asymptotics of solutions at the singular strata. This aspect has been studied 
by the recent work of Roidos and Schrohe \cite{RoSc12} in the context of manifolds with isolated conical 
singularities, which has partly motivated the present discussion here. Using the notion of maximal regularity, 
they establish short time existence of solutions to the Cahn-Hilliard equation in certain weighted Mellin-Sobolev spaces which then yields regularity and asymptotics of solutions near the conical point. \medskip

In this paper we study the semi-linear parabolic equations of fourth order in the geometric setup of spaces with incomplete 
edges, which generalizes the notion of isolated conical singularities. Our method is different from \cite{RoSc12}
and, as emphasized above, uses the microlocal construction of the heat kernel for the bi-Laplacian. \medskip

Another recent example of higher order geometric evolution equations has been studied by 
Bahuaud and Helliwell \cite{BahDyl}. The authors consider geometric flows by powers of the Laplacian
applied to the Ricci tensor or generated by the ambient obstruction tensor.
The ambient obstruction tensor was introduced by Charles Fefferman and Robin Graham \cite{FG85}
as the obstruction to a formal expansion of an asymptotically hyperbolic Einstein metric with 
a given conformal infinity in dimension $n + 1$. When $n = 4$, the ambient obstruction tensor is the Bach tensor.
\medskip

In both instances the geometric flows admit a strongly parabolic linearization after some de Turck like 
adjustment by the Lie derivative of the metric with respect to a suitable vector field. \medskip 

The microlocal analysis of the biharmonic heat kernel on edge spaces, 
presented here, allows for derivation of 
Schauder estimates with respect to certain H\"older spaces and 
ultimately leads to short time existence results 
for fourth order PDE's, including the geometric flows studied 
in \cite{BahDyl} in the setting of singular spaces. 
This will be the subject of forthcoming analysis.\medskip

We also point out that our approach is not limited to squares of Laplacians on functions, but yields 
similar results for general powers of Hodge Laplacians on differential forms along the same lines.
\bigskip

\emph{Acknowledgements}
The author would like to thank Elmar Schrohe, Rafe Mazzeo and 
Eric Bahuaud for insightful discussions and gratefull acknowledges 
the support by the Hausdorff Research Institute at the University of Bonn.

\section{Preliminaries and statement of the main results}

In this section we introduce the notion of manifolds with edge singularities, specify
a self-adjoint extension of the bi-Laplacian and state our main results.

\subsection{Manifolds with incomplete edge singularities}\label{edge-section}

We introduce the fundamental geometric aspects of spaces with incomplete edge singularities,
as described in detail in \cite{Maz:ETD}, compare also \cite{MazVer:ATM}. \medskip

Let $\overline{M}$ 
be a compact stratified space with its open interior $M$ as a single top-dimensional 
stratum, and a single lower dimensional stratum $B$, which is a smooth closed manifold by definition
of stratified spaces. The singular stratum $B$ admits an open neighborhood $U\subset \overline{M}$
and a radial function $x:U\cap M \to \R$, such that $U\cap M$ is a smooth fibre
bundle over $B$ with fibre $\mathscr{C}(F)=(0,1)\times F$, a finite open cone over a
compact smooth manifold $F$. The restriction of $x$ to each fibre defines a radial function of that cone.
\medskip

The singular stratum $B$ in $\overline{M}$ may be \emph{resolved} and defines 
a compact manifold $\widetilde{M}$ with boundary $\partial M$, 
where $\partial M$ is the total space of a fibration $\phi: \partial M \to B$ with the fibre $F$. The resolution 
process is described in detail for instance in \cite{Maz:ETD}. The neighborhood $U$ lifts to a collar neighborhood 
$\U$ of the boundary, which is again a smooth fibration over $B$ with fibre $[0,1)\times F$, a 
cylinder with the radial function $x$.

\begin{defn}\label{d-edge}
A Riemannian manifold $(M,g)$ with an incomplete edge singularity at $B$ 
is the open stratum of a stratified space with a single lower dimensional stratum $B$, 
and the Riemannian metric $g$, such that $g=g_0+h$ over $\U$, where $|h|_{g_0}=O(x)$ as $x \to  0$ and
$$g_0\restriction \U\backslash \partial M=dx^2+x^2 g^F+\phi^*g^B,$$
with $g^B$ being a Riemannian metric on $B$, and
$g^F$ a symmetric 2-tensor on the fibration $\partial M$ which restricts to a 
Riemannian metric on each fibre $F$.
\end{defn}

We set $m=\dim M, b=\dim B$ and $f=\dim F$. Clearly, $m=1+b+f$. 
We assume henceforth $f= \dim F \geq 1$. Otherwise $M$
reduces to a compact manifold with boundary, where our discussion below is no longer applicable.
\medskip

Similarly to other discussions in the singular edge setup, see 
\cite{Alb:RIT}, \cite{BDV:THE},\cite{BahVer:YFO} and \cite{MazVer:ATM}, we additionally 
require $\phi: (\partial M, g^F + \phi^*g^B) \to (B, g^B)$ to be a Riemannian submersion. 
If $p\in \partial M$, then the tangent bundle $T_p\partial M$ splits into vertical and horizontal subspaces as 
$T^V_p \partial M \oplus T^H_p \partial M$, where $T^V_p\partial M$ is the tangent space to the fibre of 
$\phi$ through $p$ and $T^H_p \partial M$ is the annihilator of the subbundle 
$T^V_p\partial M \lrcorner g^F \subset T^*\partial M$ ($\lrcorner$ meaning contraction).  
The requirement for $\phi$ to be a Riemannian submersion is the condition that the restriction of the 
tensor $g^F$ to $T^H_p \partial M$ vanishes. \medskip

We summarize the necessary assumptions on $g$
in the following definition.

\begin{defn}\label{def-admissible}
Let $(M,g)$ be a Riemannian manifold with an incomplete edge singularity. 
The Riemannian metric $g=g_0+h$ is said to be admissible if 
\begin{enumerate}
\item the fibration $\phi: (\partial M, g^F + \phi^*g^B) \to (B, g^B)$ is a Riemannian submersion,
\item the Laplace Beltrami operators $\Delta_{F, y}$ 
associated to $(F,g^F |_{\phi^{-1}(y)})$ \\ for any $y\in B$ are isospectral. 
\item the lowest non-zero eigenvalue $\lambda_0 > 0$ of the 
Laplace Beltrami \\ operators $\Delta_{F, y}$ satisfies $\lambda_0 > \dim F$.
\item $h$ vanishes to second order at $x=0$, i.e. $|h|_{g_0}=O(x^2)$ as $x\to 0$.
\end{enumerate}
\end{defn} 

The reasons behind the feasibility assumptions are as follows. 
Let $y=(y_1,...,y_{b})$ be the local coordinates on 
$B$ lifted to $\partial M$ and then extended inwards. 
Let $z=(z_1,...,z_f)$ restrict to local coordinates on $F$ 
along each fibre. Then $(x,y,z)$ are the local coordinates on $M$ near the boundary. Cons?ider the Laplace Beltrami operator $\Delta$ associated to $(M,g)$ and its normal operator 
$N(x^2\Delta)_{y_0}$, defined as the limiting operator with respect to the local family of dilatations 
$(x,y,z) \to (\lambda x, \lambda (y-y_0), z)$ and acting on functions 
on the model edge $\R^+_s \times F \times \R^b_u$. Under the first admissibility assumption, 
$N(x^2\Delta)_{y_0}$ is naturally identified with $s^2$ times the Laplace Beltrami operator on 
the model edge $\R^+_s \times F \times \R^b_u$ with incomplete edge metric $g_{\textup{ie}} = ds^2 + s^2 g^F + |du|^2$. This is key for constructing the initial crude approximation of 
the biharmonic heat kernel. 
\medskip

The second condition on isospectrality is severe, but has to be 
imposed to ensure polyhomogeneity of the associated heat kernels 
when lifted to the corresponding parabolic blowup space. More precisely 
we actually only need that the eigenvalues of the Laplacians on fibres 
are constant in a fixed range $[0,1]$, though we still make the stronger assumption 
for a clear and convenient representation. \medskip

The reasons behind the last two admissibility assumptions are 
of technical rather than geometric nature and somewhat 
less straightforward to explain. However we point out that condition 
$\lambda_0 > \dim F$ is easily satisfied by a rescaling of $g^F$. 
Condition $|h|_{g_0}=O(x^2)$ in particular 
holds for \emph{even} metrics which depend on $x^2$ instead of $x$. 
Altogether the admissibility assumptions yield precise information 
on the heat kernel expansion, which is then used in 
Proposition \ref{heat}. \medskip

\subsection{Edge vector fields and Banach spaces}

An important ingredient in the analysis of singular edge spaces is the vector space $\V$ 
of \emph{edge} vector fields smooth in the interior of $\widetilde{M}$ and tangent at the boundary $\partial M$ to the fibres of the fibration. 
This space $\V$ is closed under the ordinary Lie bracket of vector fields, hence defines a Lie algebra. Its 
description in local coordinates is as follows. Consider the local coordinates $(x,y,z)$ on $\overline{M}$ near the boundary.
Then the edge vector fields $\V$ are locally generated by 
\[
\left\{x\partial_x, x\partial_{y_1}, \dots, x\partial_{y_b}, 
\partial_{z_1},\dots, \partial_{z_f}\right\}.
\]

We may now define the Banach space of continuous sections 
$\mathscr{C}_{\textup{ie}}^0(M)$, continuous on $\widetilde{M}$ 
up to the boundary and fibrewise constant at $x=0$. 
This is precisely the space of continuous sections with respect to the topology on $M$ 
induced by the Riemannian metric $g$. The standard 
space of $2k$-times continuously differentiable functions 
in the open interior $M$ is denoted by $C^{2k}(M)$. 
Banach spaces of higher order are now defined as 
follows, compare \cite{BDV:THE}.

\begin{defn}\label{defn:norms}
Let $(M,g)$ be a Riemannian manifold with an incomplete edge metric. 
Let $\mathcal{D}$ denote a subspace generated by a finite collection 
$\widehat{\mathcal{D}}$ of derivatives 
in $\{\Delta, x^{-1}\VV, x^{-1}\V, \V\}$, which will be specified later.
Then for each $k\in \N$ we define
\begin{align*}
&\ck := \{u \in C^{2k}(M) \cap \mathscr{C}_{\textup{ie}}^0(M) \mid  
X \circ \Delta^j u \in \mathscr{C}_{\textup{ie}}^0(M), 
X\in \mathcal{D}, \ j<k\},
\\ &\textup{with the norm } \ \|u\|_{2k} := \|u\|_\infty + \sum_{j=0}^k 
\sum_{X\in \widehat{\mathcal{D}}} \|X \circ \Delta^j u\|_\infty.
\end{align*}
\end{defn}

\subsection{Self-adjoint extension of the bi-Laplacian}\label{friedrichs-subsection}

Let $\Delta$ denote the Laplace Beltrami operator acting on functions on an
incomplete edge space $(M,g)$ with an admissible incomplete edge metric $g$. 
Consider the space of square-integrable forms $L^2(M,g)$, with respect to $g$. The \emph{maximal} and 
\emph{minimal} closed extensions of $\Delta$ are defined by the domains
\begin{equation}\label{expansion-w}
\begin{split}
\mathcal{D}_{\max}(\Delta) &:= \{ u\in L^2(M,g) \mid  \Delta u \in L^2(M,g) \}, \\ 
\mathcal{D}_{\min}(\Delta) &:= \{ u \in \mathcal{D}_{\max}(\Delta)  \mid \exists\,  u_j \in C^\infty_0(M)\ \mbox{such that}  \\ 
&u_j \to u\ \mbox{and}\ \Delta u_j \to \Delta u\ \mbox{both in}\ L^2(M,g) \}.
\end{split}
\end{equation}
where $\Delta u\in L^2$ is a priori understood in the distributional sense. Under the 
unitary rescaling transformation in the singular edge neighborhood
\begin{align}\label{rescaling}
\Phi:L^2(\U, \textup{dvol}(g)) \to L^2(\U, x^{-f}\textup{dvol}(g)), \ u \mapsto x^{f/2} u,
\end{align}
the rescaled Laplacian $\Delta^\Phi = \Phi \circ \Delta \circ \Phi^{-1}$ is a perturbation of 
$$
\Delta^\Phi_0 = -\frac{\partial^2}{\partial x^2} + 
\frac{1}{x^2}\left(\Delta_{F,y} + \left(\frac{f-1}{2}\right)^2 - \frac{1}{4}\right),
$$
with higher order terms coming from the curvature of the Riemannian submersion 
$\phi$ and the second fundamental forms of the fibres $F$. \medskip

The following lemma is a straightforward reformulation of \cite[Lemma 2.2]{MazVer:ATM}
for the simpler case of Laplace Beltrami operators.

\begin{lemma}[\cite{MazVer:ATM}]\label{max}
Let $(M,g)$ be an incomplete edge space with an admissible edge metric. Consider 
the increasing sequence of eigenvalues $(\sigma_j)_{j\in \N}$ of $\Delta_{F,y}$, 
counted with their multiplicities, and put $\nu_j^2:= \sigma_j + (f-1)^2/4$. 
The associated indicial roots are given by $\gamma_j^{\pm} = \pm \nu_j + 1/2$. 
Let $p\in \N$ be the largest index such that $\nu_p \in [0,1)$. 
Then any $u \in \mathcal{D}_{\max}(\Delta)$ admits a weak expansion as $x\to 0$, 
in the sense that for any test function $\chi \in C^\infty(B)$ there is an expansion of the pairing
\begin{equation*}
\int_B \Phi u(x,y,z) \chi(y)\, dy \, \sim \, \sum_{j=1}^p \left( \psi_j^+(x)  c_j^+[u,\chi](z) 
+ \psi_j^-(x)  c^-_j[u,\chi](z) \right) + \tilde{u}, \ x \to 0,
\end{equation*}
where the higher order term $\tilde{u} =O(x^{3/2})$ as $x\to 0$, 
and the coefficients $c_j^\pm[u,\chi]$ 
are constant for $j=1,..., \dim H^0(F)$. Moreover 
\begin{align*}
&\psi^+_j(x,z;y)=x^{\gamma_j^+}, \ \textup{and} \\
&\psi^-_j(x,z;y) = x^{1/2}(\log x), \ \textup{if} \ \nu_j=0, \\
&\psi^-_j(x,z;y) = x^{\gamma_j^-}(1+a_jx)  \ \textup{if} \ \nu_j >0,
\end{align*}
with $a_j\in \R$ uniquely determined by $\Delta$. 
\end{lemma}
The Friedrichs self-adjoint extension of $\Delta$ has been identified in \cite{MazVer:ATM} as
\begin{align}\label{F}
\mathcal{D}(\Delta_{\mathscr{F}}) = \{u \in \mathcal{D}_{\max}(\Delta) \mid \forall_{j=1,..,p}: c^-_j[u]=0\}.
\end{align}
Note that the sequence of eigenvalues $(\sigma_j)_{j=1}^p$ of $\Delta_{F,y}$ starts with $\sigma_j=0$
for $j=1,..,\dim H^0(F)$. The corresponding indicial roots compute to $\gamma_j=1/2\pm (f-1)/2$ 
and the coefficients are constant in $z\in F$, being simply the harmonic functions of fibres $F$. 
Consequently, the Friedrichs domain contains precisely those elements in the 
maximal domain whose leading 
term in the weak expansion as $x\to 0$ is given by $x^0$ with 
fibrewise constant coefficients. In particular 
\begin{align}\label{max-F}
\mathcal{D}_{\max}(\Delta) \cap  \mathscr{C}_{\textup{ie}}^0(M) \subset \mathcal{D}(\Delta_{\mathscr{F}}).
\end{align}

We fix a self-adjoint extension of the bi-Laplacian as the square of $\Delta_{\mathscr{F}}$
\begin{align}
\mathcal{D}(\Delta^2_{\mathscr{F}}) = \{u \in \Delta_{\mathscr{F}} \mid  \Delta u \in \Delta_{\mathscr{F}}\}.
\end{align}

\subsection{The biharmonic heat space blowup}

Consider $\Delta^2_{\mathscr{F}}$ and the corresponding 
heat operator $e^{-t\Delta^2_{\mathscr{F}}}$. Let $H$ be the biharmonic heat kernel, 
the Schwartz kernel of $e^{-t\Delta^2_{\mathscr{F}}}$. 
$H$ is a priori a function on $M^2_h=\R^+\times \widetilde{M}^2$.  
Let $(x,y,z)$ and $(\widetilde{x}, \wy, \widetilde{z})$ be the coordinates on the two copies of $M$ near the edge.
Then the local coordinates near the corner in $M^2_h$ are given by $(t, (x,y,z), (\widetilde{x}, \wy, \widetilde{z}))$.
The kernel $H (t, (x,y,z), (\wx,\wy,\wz))$ has a non-uniform behaviour at the submanifolds
\begin{align*}
A &=\{ (t=0, (0,y,z), (0, \wy, \wz))\in \R^+ \times \Mdel^2 \mid y= \wy\},\\
D &=\{ (t=0, (x,y,z), (\wx, \wy, \wz))\in \R^+ \times \widetilde{M}^2 \mid x=\wx, \, y= \wy, \, z=\wz\}.
\end{align*}
Exactly as in the case of the Hodge Laplacian on edges, see \cite{MazVer:ATM}, we introduce 
an appropriate blowup of the heat space $M^2_h$, such that the corresponding heat kernel lifts to a polyhomogeneous distribution 
in the sense of the definition below. This procedure has been introduced by Melrose in \cite{Mel:TAP}. 
For self-containment of the paper we repeat the definition of polyhomogeneity 
as well as the blowup process here.

\begin{defn}\label{phg}
Let $\mathfrak{W}$ be a manifold with corners, with all boundary faces embedded, and $\{(H_i,\rho_i)\}_{i=1}^N$ an enumeration 
of its boundaries and the corresponding defining functions. For any multi-index $b= (b_1,
\ldots, b_N)\in \C^N$ we write $\rho^b = \rho_1^{b_1} \ldots \rho_N^{b_N}$.  Denote by $\mathcal{V}_b(\mathfrak{W})$ the space
of smooth vector fields on $\mathfrak{W}$ which lie
tangent to all boundary faces. A distribution $\w$ on $\mathfrak{W}$ is said to be conormal,
if $\w\in \rho^b L^\infty(\mathfrak{W})$ for some $b\in \C^N$ and $V_1 \ldots V_\ell \w \in \rho^b L^\infty(\mathfrak{W})$
for all $V_j \in \mathcal{V}_b(\mathfrak{W})$ and for every $\ell \geq 0$. An index set 
$E_i = \{(\gamma,p)\} \subset {\mathbb C} \times {\mathbb N}$ 
satisfies the following hypotheses:

\begin{enumerate}
\item $\textup{Re}(\gamma)$ accumulates only at plus infinity,
\item For each $\gamma$ there is $P_{\gamma}\in \N_0$, such 
that $(\gamma,p)\in E_i$ for every $0\leq p \leq P_\gamma$,
\item If $(\gamma,p) \in E_i$, then $(\gamma+j,p') \in E_i$ for all $j \in {\mathbb N}$ and $0 \leq p' \leq p$. 
\end{enumerate}
An index family $E = (E_1, \ldots, E_N)$ is an $N$-tuple of index sets. 
Finally, we say that a conormal distribution $\w$ is polyhomogeneous on $\mathfrak{W}$ 
with index family $E$, we write $\w\in \mathscr{A}_{\textup{phg}}^E(\mathfrak{W})$, 
if $\w$ is conormal and if in addition, near each $H_i$, 
\[
\w \sim \sum_{(\gamma,p) \in E_i} a_{\gamma,p} \rho_i^{\gamma} (\log \rho_i)^p, \ 
\textup{as} \ \rho_i\to 0,
\]
with coefficients $a_{\gamma,p}$ conormal on $H_i$, polyhomogeneous with index $E_j$
at any $H_i\cap H_j$. 
\end{defn}

Our analysis of the biharmonic heat kernel will start with a discussion of an
explicitly solvable model situation, which leads to a homogeneity property \eqref{homogeneity}.
That property contains the information how precisely the 
submanifolds $A,D\subset M^2_h$ need to be blown up such 
that the heat kernel $H$ becomes polyhomogeneous.
To get the correct blowup of $M^2_h$ we first bi-parabolically 
($t^{1/4}$ is viewed as a coordinate function) blow 
up the submanifold 
\begin{align*}
A = \{ (t, (0,y, z), (0, \wy, \wz))\in \R^+ \times \Mdel^2 : t=0, y=\wy \} \subset M^2_h.
\end{align*}

The resulting heat-space $[M^2_h, A]$ is defined as the union of
$M^2_h\backslash A$ with the interior spherical normal bundle of $A$ in $M^2_h$. 
The blowup $[M^2_h, A]$ is endowed with the unique minimal differential structure 
with respect to which smooth functions in the interior of $M^2_h$ and polar coordinates 
on $M^2_h$ around $A$ are smooth. As in \cite{MazVer:ATM}, this blowup introduces four new boundary 
hypersurfaces;  we denote these by ff (the front face), rf (the right face), 
lf (the left face) and tf (the temporal face).  \medskip

The actual heat-space blowup $\mathscr{M}^2_h$ is obtained by a bi-parabolic blowup 
of $[M^2_h, A]$ along the diagonal $D$, lifted to a submanifold of $[M^2_h, A]$. The resulting blowup $\mathscr{M}^2_h$ is 
defined as before by cutting out the submanifold and replacing it with its spherical 
normal bundle. It is a manifold with boundaries and corners, visualized in Figure below.

\begin{figure}[h]
\begin{center}
\begin{tikzpicture}
\draw (0,0.7) -- (0,2);
\draw(-0.7,-0.5) -- (-2,-1);
\draw (0.7,-0.5) -- (2,-1);

\draw (0,0.7) .. controls (-0.5,0.6) and (-0.7,0) .. (-0.7,-0.5);
\draw (0,0.7) .. controls (0.5,0.6) and (0.7,0) .. (0.7,-0.5);
\draw (-0.7,-0.5) .. controls (-0.5,-0.6) and (-0.4,-0.7) .. (-0.3,-0.7);
\draw (0.7,-0.5) .. controls (0.5,-0.6) and (0.4,-0.7) .. (0.3,-0.7);

\draw (-0.3,-0.7) .. controls (-0.3,-0.3) and (0.3,-0.3) .. (0.3,-0.7);
\draw (-0.3,-1.4) .. controls (-0.3,-1) and (0.3,-1) .. (0.3,-1.4);

\draw (0.3,-0.7) -- (0.3,-1.4);
\draw (-0.3,-0.7) -- (-0.3,-1.4);

\node at (1.2,0.7) {\large{rf}};
\node at (-1.2,0.7) {\large{lf}};
\node at (1.1, -1.2) {\large{tf}};
\node at (-1.1, -1.2) {\large{tf}};
\node at (0, -1.7) {\large{td}};
\node at (0,0.1) {\large{ff}};
\node at (-3,-1) {$\mathscr{M}^2_h$};

\draw[dotted, thick] (5,0) -- (5,2);
\node at (5.6,1.7) {$t^{1/4}$};
\draw[dotted, thick] (5,0) -- (3,-1);
\node at (3.2,-0.5) {$x$};
\draw[dotted, thick] (5,0) -- (7,-1);
\node at (6.8,-0.5) {$\wx$};
\node at (5,-1) {$M^2_h$};

\node at (2.5,0.7) {$\longrightarrow$};
\node at (2.5,1) {$\beta$};

\end{tikzpicture}
\end{center}
\label{figure-incomplete}
\caption{The biharmonic heat-space blowup $\mathscr{M}^2_h$.}
\end{figure}
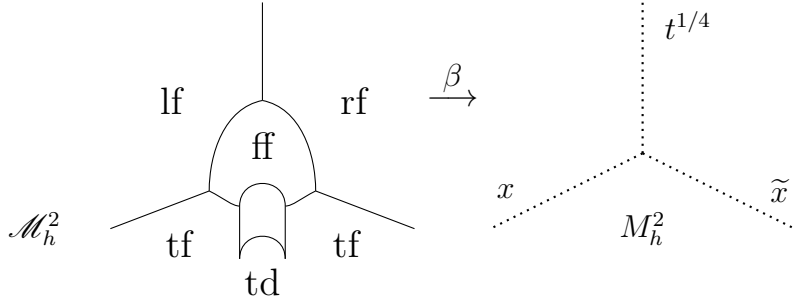

The projective coordinates on $\mathscr{M}^2_h$ are then given as follows. 
Near the top corner of the front face ff, the projective coordinates are given by
\begin{align}\label{top-coord}
\rho=t^{1/4}, \  \xi=\frac{x}{\rho}, \ \widetilde{\xi}=\frac{\wx}{\rho}, \ u=\frac{y-\wy}{\rho}, \ z, \ \wy, \ \wz,
\end{align}
where in these coordinates $\rho, \xi, \widetilde{\xi}$ are the defining 
functions of the boundary faces ff, rf and lf respectively. 
For the bottom corner of the front face near the right hand side projective coordinates are given by
\begin{align}\label{right-coord}
\tau=\frac{t}{\wx^4}, \ s=\frac{x}{\wx}, \ u=\frac{y-\wy}{\wx}, \ z, \ \wx, \ \wy, \ \widetilde{z},
\end{align}
where in these coordinates $\tau, s, \widetilde{x}$ are
 the defining functions of tf, rf and ff respectively. 
For the bottom corner of the front face near the left hand side
projective coordinates are obtained by interchanging 
the roles of $x$ and $\widetilde{x}$. Projective coordinates 
on $\mathscr{M}^2_h$ near temporal diagonal are given by 
\begin{align}\label{d-coord}
\eta=\frac{t^{1/4}}{\wx}, \ S =\frac{(x-\wx)}{t^{1/4}}, \ 
U= \frac{y-\wy}{t^{1/4}}, \ Z =\frac{\wx (z-\wz)}{t^{1/4}}, \  \wx, \ 
\wy, \ \widetilde{z}.
\end{align}
In these coordinates tf is the face in the limit $|(S, U, Z)|\to \infty$, 
ff and td are defined by $\widetilde{x}, \eta$, respectively. 
The blowdown map $\beta: \mathscr{M}^2_h\to M^2_h$ is in 
local coordinates simply the coordinate change back to 
$(t, (x,y, z), (\widetilde{x},\wy, \widetilde{z}))$. \medskip

\subsection{Statement of the main results}

Our first main result is concerned with the asymptotic properties 
of the biharmonic heat kernel as a polyhomogeneous function on the
biharmonic heat space blowup.

\begin{thm}
Let $(M^m,g)$ be an incomplete edge space with an admissible edge metric $g$,
and let $\Delta_{\mathscr{F}}$ denote the Friedrichs extension of the 
corresponding Laplace-Beltrami operator. Let $H$ be the Schwartz kernel
of the heat operator $e^{-t\Delta^2_{\mathscr{F}}}$ associated to the bi-Laplacian
$\Delta^2_{\mathscr{F}}$. Then the lift $\beta^*H$ is polyhomogeneous on $\mathscr{M}^2_h$ of order 
$(-\dim M)$ at $\ff$ and $\td$, vanishing to infinite order at $\tf$, and with the
index set at $\rf$ and $\lf$ given by $E+\N_0$ where
$$
E=\{\gamma \geq 0 \mid  \gamma= -\frac{(f-1)}{2}+ \sqrt{\frac{(f-1)^2}{4}+ \sigma^2},
 \, \sigma^2 \in \textup{Spec}\, \Delta_{F,y}\}.
$$
More precisely, if $s$ denotes the boundary defining function of $\rf$, we obtain 
\begin{align*}
\beta^* H \sim \sum_{\gamma \in E} \left(\sum_{j=0}^\infty s^{\gamma+2j} a_{\gamma,j}(\beta^* H) 
+ \sum_{j=0}^\infty s^{\gamma+2+j} a'_{\gamma,j}(\beta^* H)\right)\ \textup{as} \ s\to 0,
\end{align*}
where the coefficients $a_{\gamma,j}(H)$ are of order $(-m)$ at the front face 
and lie in their corresponding $\Delta_{F,y}$ eigenspaces. The higher coefficients $a'_{\gamma,j}(\beta^* H)$
are of order $(-m+1)$ at ff.
\end{thm}

We employ this microlocal heat kernel description to establish our next 
main theorem on the mapping properties of the corresponding biharmonic
heat operator.

\begin{thm}
Let $(M^m,g)$ be an incomplete edge space with an admissible edge metric $g$
and $\Delta$ the corresponding Laplace Beltrami operator. 
Put $\mathcal{D}_0=\langle \Delta \rangle$ and $\mathcal{D}=
\langle\Delta,  x^{-1}\VV, x^{-1}\mathcal{V}'_e, \V \rangle$, 
where $\mathcal{V}'_e \subset \V$ consists locally of all edge vector fields where $x\partial_y$
is weighted with functions that are fibrewise constant. Then the biharmonic heat operator 
$e^{-t\Delta_{\mathscr{F}}^2}$
is a bounded map between the (weighted) Banach spaces
$$e^{-t\Delta_{\mathscr{F}}^2}: \mathscr{C}^{2k}_{\textup{ie}}(M, \mathcal{D}_0) 
\to t^{-1/4} \mathscr{C}^{2(k+1)}_{\textup{ie}}(M, \mathcal{D}).$$
\end{thm}

We should point out that in the theorem above the target Banach space
$\mathscr{C}^{2(k+1)}_{\textup{ie}}(M, \mathcal{D})$ is defined with respect to 
the set $\mathcal{D}$ that includes a large variety of higher order derivatives.
In fact, from the perspective of the presented proof, 
this is the largest possible variety of derivatives with respect to which 
boundedness of the biharmonic heat operator $e^{-t\Delta_{\mathscr{F}}^2}$ persists.
On the other hand, the initial space $\mathscr{C}^{2k}_{\textup{ie}}(M, \mathcal{D}_0)$
poses significantly less regularity assumptions, since it is defined with respect to 
a very restricted set of derivatives $\mathcal{D}_0=\langle\Delta\rangle$. 
In that respect, the biharmonic heat operator indeed improves regularity.\medskip

An important aspect of the statement is that regularity is not defined with 
respect to $x^{-1}\V$ but rather $x^{-1}\mathcal{V}'_e$, i.e. we require the 
generators to be weighted with functions that are constant on fibres $F$ when 
restricted to $x=0$. This is due to the fact that we consider spaces of continuously differentiable
functions with the continuity defined with respect to the Riemannian metric $g$.
Such continuous functions are constant on fibres at $\partial M$. For this aspect 
also note the Remark \ref{ie}.\medskip

Our final result is concerned with local existence of solutions to certain
semi-linear parabolic equations of fourth order.

\begin{thm}
Let $(M,g)$ be an incomplete edge space with an admissible edge metric $g$. 
Put $\mathcal{D}'=\langle\Delta,  \VV, \V \rangle$ and $\mathcal{D}=\langle\Delta,  x^{-1}\VV, x^{-1}\mathcal{V}'_e, \V\rangle$, 
where $\mathcal{V}'_e \subset \V$ consists locally of linear combinations of $\{x\partial_x, x\partial_y, \partial_z\}$, 
where $x\partial_y$ is weighted with functions that are fibrewise constant. Assume $Q: \mathscr{C}^{2(k+1)}_{\textup{ie}}(M, \mathcal{D}')
\to \mathscr{C}^{2k}_{\textup{ie}}(M, \mathcal{D}')$ is locally Lipschitz. Then the semilinear equation
\begin{equation*}
\partial_t u + \Delta^2u = Q(u), \ u(0) = u_0 \in \mathscr{C}^{2(k+1)}_{\textup{ie}}(M, \mathcal{D}')
\end{equation*}
has a unique solution $u \in C([0,T], \mathscr{C}^{2(k+1)}_{\textup{ie}}(M, \mathcal{D}))\cap C^\infty((0,T]\times M)$, for some $T > 0$,
where $T$ may be estimated from below in terms of $||u_0||_{2(k+1)}$.
\end{thm}

As an application we arrive at a statement on existence and regularity of 
solutions to the Cahn-Hilliard equation.

\begin{cor}
Let $(M,g)$ be an incomplete edge space with an admissible edge metric $g$. 
Put $\mathcal{D}'=\langle \Delta,  \VV, \V \rangle$ and 
$\mathcal{D}=\langle\Delta,  x^{-1}\VV, x^{-1}\mathcal{V}'_e, \V\rangle$, 
where $\mathcal{V}'_e \subset \V$ consists locally of linear combinations of $\{x\partial_x, x\partial_y, \partial_z\}$, 
where $x\partial_y$ is weighted with functions that are fibrewise constant. Then the Cahn-Hilliard equation
\begin{equation*}
\partial_t u + \Delta^2u + \Delta (u-u^3)= 0, \ u(0) = u_0 \in \mathscr{C}^{2(k+1)}_{\textup{ie}}(M, \mathcal{D}')
\end{equation*}
has a unique solution $u \in C([0,T], \ckk)\cap C^\infty((0,T]\times M)$, for some $T > 0$.
\end{cor} 

It should be noted that in correspondence with \cite{RoSc12} our approach 
leads to an explicit identification of the asymptotics of the Cahn-Hilliard solution 
at $x=0$. Indeed, $u\in \ckk \subset\dom (\Delta_{\mathscr{F}}^{k+1})$, which 
yields a partial asymptotics of $u$ to higher and higher order, depending on $k\in \N$, by an iterative application of Lemma \ref{max} for $k$ steps.

\section{Microlocal heat kernel construction}\label{heat-section}

\subsection{Biharmonic heat kernel on a model edge}

In this section we construct the heat kernel for $\Delta^2_{\mathscr{F}}$ explicitly.
We begin with studying the homogeneity properties of the heat kernel for the bi-Laplacian 
in the model case of an exact edge $(\calE=\R^b \times \mathscr{C}(F), dy^2+ g)$
where $(\mathscr{C}(F) = (0,\infty) \times F, g= ds^2+s^2g^F)$ is an exact unbounded 
cone over a closed Riemannian manifold $(F,g^F)$. The Laplacian $\Delta_{\calE}$ 
on the exact edge is then a sum of the Laplacian on $(\mathscr{C}(F), g)$ and the Euclidean 
Laplacian on $\R^b$. Consider the scaling operation ($\lambda > 0$)

\begin{align*}
&\Psi_\lambda: C^\infty (\R^+ \times \calE^2) \to C^\infty (\R^+ \times \calE^2), \\
&(\Psi_\lambda u)(t,(s,y,z),(\widetilde{s},\wy,\widetilde{z})) = u(\lambda^4 t,(\lambda s,\lambda(y-\wy), z),
(\lambda\widetilde{s},\lambda \wy,\widetilde{z})).
\end{align*}
Under the scaling operation we find 
\begin{align}\label{scaling-heat-eqn}
(\partial_t + \Delta_\calE^2)\Psi_\lambda u = 
\lambda^4 \Psi_\lambda (\partial_t + \Delta_\calE^2)u.
\end{align}
Consequently, given the heat kernel $H_\calE$ for the Friedrichs extension 
of $\Delta_\calE^2$ (or at that stage any other self-adjoint extension), any  
multiple of $\Psi_\lambda H_\calE$ still solves the heat equation and also maps into the domain 
of $\Delta_\calE^2$. For the initial condition we obtain substituting 
$\widetilde{Y}=\lambda \wy, \widetilde{S}=\lambda \widetilde{s}$
\begin{align*}
&\lim_{t\to 0} \int_\calE (\Psi_\lambda H_\calE)
(t,\, s,\, y,\, z,\, \widetilde{s},\, \wy, \, \widetilde{z}) u(\widetilde{s}, \, \wy,\,  \widetilde{z})\,
\widetilde{s}^f \, d\widetilde{s}\,  d\wy \, d\widetilde{z} \\
=& \lim_{t\to 0} \lambda^{-1-b-f}\int_\calE H_\calE
(\lambda ^4 t,\, \lambda s, \, \lambda y - \widetilde{Y}, \, z,\, \widetilde{S}, \, \widetilde{Y},\, \widetilde{z}) \, 
u(\widetilde{S}/\lambda, \, \widetilde{Y}/\lambda, \, \widetilde{z}) \widetilde{S}^f \, d\widetilde{S}\, d\widetilde{Y} d\widetilde{z} \\ 
=& \lambda^{-1-b-f} u(\lambda s/\lambda, \,\lambda y/\lambda, \, z) = \lambda^{-1-b-f} u(s,\, y,\, z).
\end{align*} 
By uniqueness of the heat kernel we obtain
\begin{align}\label{homogeneity}
\Psi_\lambda H_\calE= \lambda^{1-b-f}H_\calE.
\end{align}

In addition to the homogeneity properties of $H_\calE$, we also require a full asymptotic 
expansion of the biharmonic heat kernel as $(s,\widetilde{s})\to 0$. 
We accomplish this by establishing an explicit integral representation of $H_\calE$.
Under the unitary rescaling \eqref{rescaling} and a spectral decomposition of 
$L^2(F, g^F)$ into $\sigma^2$-eigenspaces of $\Delta_F$, we may write for the rescaled model edge Laplacian
\begin{align*}
\Delta^\Phi_\calE &= -\partial_s^2 + s^{-2} \left(\Delta_F + \left(\frac{f-1}{2}\right)^2-\frac{1}{4}\right) + \Delta_{\R^b}\\
&=\bigoplus_{\sigma} -\partial_s^2 + s^{-2} \left(\sigma^2 + \left(\frac{f-1}{2}\right)^2-\frac{1}{4}\right) + \Delta_{\R^b}
=:\bigoplus_{\sigma} l_{\nu(\sigma)} + \Delta_{\R^b},
\end{align*}
where $\nu(\sigma):= \sqrt{\sigma^2+ (f-1)^2/4}$ and $l_{\nu(\sigma)}$ is defined on $C^\infty_0(0,\infty)$. 
The Friedrichs extension of $\Delta_\calE$ 
is compatible with the decomposition, compare a similar discussion in (\cite{Ver:FDF}, Proposition 4.9).
As a special case of Lemma \ref{max}, $l_\nu$ has unique self-adjoint extension $L_\nu$ in $L^2(\R^+)$ for $\nu \geq 1$, and 
in case of $\nu \in [0,1)$, solutions $u \in \dom(l_{\nu, \max})$ admit a partial asymptotic expansion as $s\to 0$
\begin{equation*}
u(s) = \widetilde{u} + c^+[u] \, s^{\nu+1/2} + c^-[u] \left\{ \begin{split} &s^{-\nu+1/2}, \ \nu \in (0,1), \\
&\sqrt{s}\log(s), \ \nu =0, \end{split}\right. \qquad  \widetilde{u} \in \dom (l_{\nu, \min}).
\end{equation*}
Then the Friedrichs extension $L_\nu$ of $l_\nu$ is defined, similar to \eqref{F},
by requiring $c^-[u]=0$, and moreover, identifying $\Delta_{\R^b}$ with its unique self-adjoint extension in $L^2(\R^b)$,
we may write
\begin{align}\label{decomp}
\Delta^\Phi_{\calE,\mathscr{F}} = \bigoplus_{\sigma} L_{\nu(\sigma)} + \Delta_{\R^b}.
\end{align}
Consequently, it suffices to construct the biharmonic heat kernel for $L_{\nu} + \Delta_{\R^b}$ in $L^2(\R^+\times \R^b)$.
Denote by $J_\nu$ the $\nu$-th Bessel function of first kind and consider the Hankel transform of order $\nu\geq 0$
\begin{align}
(\mathscr{H}_\nu u)(s) := \int_0^\infty \sqrt{s s'} J_\nu(s s') u(s') ds', \ u \in C^\infty_0(0,\infty).
\end{align}
By (\cite{Co:LTD}, Chapter III) and also by (\cite{Les:OFT}, Proposition 2.3.4), the Hankel transform 
extends to a self-adjoint isometry on $L^2(\R^+)$. We denote by 
$$(\mathscr{F} u)(\xi) := (2\pi)^{-b/2} \int_{\R^b} u(y) e^{-i y\cdot \xi} dy, \quad u \in C^\infty_0(\R^b),$$ the Fourier transform
on $\R^b$, which extends to an isometric automorphism of $L^2(\R^b)$. Consequently, 
$\mathscr{G}_\nu:=\mathscr{H}_\nu \circ \mathscr{F}$ defines an isometric automorphism of 
$L^2(\R^+\times \R^b)$ such that $\mathscr{G}_\nu^{-1} = \mathscr{H}_\nu \circ \mathscr{F}^{-1}$.
Applying (\cite{Les:OFT}, Proposition 2.3.5), we arrive at the following
\begin{prop}\label{diag}
The isometric automorphism $\mathscr{G}_\nu$ diagonalizes $L_\nu + \Delta_{\R^b}$.
More precisely, 
\begin{align*}
&\dom (L_\nu + \Delta_{\R^b}) = \{u \in L^2(\R^+\times \R^b) 
\mid (S^2+ |\Xi|^2) \, \mathscr{G}_\nu u \in L^2(\R^+\times \R^b)\}, \\
&\mathscr{G}_\nu \, (L_\nu + \Delta_{\R^b}) \, \mathscr{G}^{-1}_\nu = S^2+ |\Xi|^2,
\end{align*}
where $X, \Xi$ denote multiplication operators by $x\in \R^+$ and $\xi\in \R^b$, respectively.
\end{prop}

Similary, the isometry $\mathscr{G}_\nu$ diagonalizes the squared 
operator $(L_\nu + \Delta_{\R^b})^2$, identifying its action with $(S^2+ |\Xi|^2)^2$.
Consequently we may express the biharmonic heat kernel of $(L_\nu + \Delta_{\R^b})^2$
as an integral in terms of Bessel functions. For $u \in C^\infty_0(\R^+\times \R^b)$ we find
\begin{align*}
&\left(e^{-t (L_\nu + \Delta_{\R^b})^2} u \right)(s,y) = 
\left(\mathscr{G}_\nu e^{-t (S^2+ |\Xi|^2)^2} \mathscr{G}_\nu^{-1} u\right)(s,y) \\
&=(2\pi)^{-b/2} \int_{\R^b} \int_0^\infty \left(\int_{\R^b} \int_0^\infty e^{i(y-\wy)\xi} 
\, \sqrt{s\widetilde{s}} \, J_\nu(s\rho) J_\nu(\widetilde{s}\rho) \, \rho \, 
e^{-t(\rho^2+ |\xi|^2)^2} d\rho \, d\xi \right) \\ & \quad u(\widetilde{s}, \wy)\, d\widetilde{s} \, d\wy.
\end{align*}
Denote by $\phi_\sigma$ the normalized $\sigma^2$-eigenfunction of $\Delta_F$, where 
we count the eigenvalues $\sigma^2\in \textup{Spec}(\Delta_F)$ with their multiplicities.
Then, as a consequence of \eqref{decomp}, we finally obtain for the $\Phi$-rescaled
biharmonic heat kernel on a model edge 
\begin{align*}
H^\Phi_\calE = 
(2\pi)^{-b/2} \bigoplus_{\sigma} \int_{\R^b} \int_0^\infty e^{i(y-\wy)\xi} \, \sqrt{s\widetilde{s}} \, 
J_{\nu(\sigma)}(s\rho) J_{\nu(\sigma)}(\widetilde{s}\rho) \, &\rho \, e^{-t(\rho^2+ |\xi|^2)^2} d\rho \, d\xi 
\\ &\cdot \phi_\sigma (z) \otimes \phi_\sigma (\wz).
\end{align*}
The $\nu$-th Bessel function of first kind admits an asymptotic expansion for small arguments
$J_\nu(\zeta) \sim \sum_{j=0}^\infty a_j \zeta^{\nu+2j}$, as $\zeta \to 0$.
This yields an asymptotic expansion of $H^\Phi_\calE$ as $(s,\widetilde{s})\to 0$ and consequently, 
rescaling back, we obtain as $s\to 0$
\begin{align}\label{expansion-cone}
\quad H_\calE (t,\, s,\, y,\, z,\, \widetilde{s},\, \wy, \, \widetilde{z})
\sim \sum_{\gamma} a_{\nu,j}(t,\widetilde{s}, y, \wy, z,\widetilde{z}) s^{\gamma+2j},
\end{align}
where the summation is over all $\gamma = -(f-1)/2 +\sqrt{\sigma^2 + (f-1)^2/4}$
with $\sigma^2 \in \textup{Spec}\Delta_F$, counted with multiplicity, and each coefficient $a_{\nu,j}$ lies in the
corresponding $\sigma^2$-eigenspace. We summarize the properties of $H_\calE$, 
established above, in a single proposition for later reference.

\begin{prop}\label{model-heat}
Consider the model edge $(\calE=\R^b \times \mathscr{C}(F), dy^2+ g)$,
where $(\mathscr{C}(F) = (0,\infty) \times F, g= ds^2+s^2g^F)$ is an exact unbounded 
cone over a closed Riemannian manifold $(F^f,g^F)$. Fix the Friedrichs self-adjoint extension 
of the associated Laplace Beltrami operator $\Delta_\calE$. Then the biharmonic heat kernel $H_\calE$ 
of $\Delta_\calE^2$ is homogeneous of order $(-1-b-f)$ under the scaling operation ($\lambda > 0$)
\begin{align*}
&\Psi_\lambda: C^\infty (\R^+ \times \calE^2) \to C^\infty (\R^+ \times \calE^2), \\
&(\Psi_\lambda u)(t,(s,y,z),(\widetilde{s},\wy,\widetilde{z})) = u(\lambda^4 t,(\lambda s,\lambda(y-\wy), z),
(\lambda\widetilde{s},\lambda \wy,\widetilde{z})).
\end{align*}
Moreover, $H_\calE$ admits an asymptotic expansion as $(s, \widetilde{s}) \to 0$ with the
index set given by $E+2\N_0$, where
$$
E=\{\gamma \geq 0 \mid  \gamma= -\frac{(f-1)}{2}+ \sqrt{\frac{(f-1)^2}{4}+ \sigma^2},
\, \sigma^2 \in \textup{Spec}\, \Delta_{F}\},
$$
uniformly in other variables and with coefficients taking value in the corresponding 
$\sigma^2$-eigenspace.
\end{prop}

\subsection{Construction of the biharmonic heat kernel}

We can now proceed from the analysis of the heat kernel on the model edge
to the construction of the heat kernel $H$ for the bi-Laplacian on a space $(M,g)$
with an incomplete admissible edge metric. The heat kernel construction here follows 
ad verbatim the discussion in \cite{MazVer:ATM} for the edge Laplacian, with the only difference that 
for the bi-Laplacian now rather $t^{1/4}$ instead of $\sqrt{t}$ is treated as a smooth variable. \medskip

In case the edge manifold is an exact edge $(\calE=\R^b \times \mathscr{C}(F), dy^2+ g)$
where $(\mathscr{C}(F) = (0,\infty) \times F^f, g= ds^2+s^2g^F)$, 
Proposition \ref{model-heat} implies that $H_\calE$ 
lifts to a polyhomogeneous conormal distribution on the biharmonic heat space blowup, 
of order $(-m)$ at the front and the temporal diagonal faces, vanishing to infinite order at $\tf$,
and with the index set at $\rf$ and $\lf$ given by $E+2\N_0$, where
$$
E=\{\gamma \geq 0 \mid  \gamma= -\frac{(f-1)}{2}+ \sqrt{\frac{(f-1)^2}{4}+ \sigma^2},
 \, \sigma^2 \in \textup{Spec}\, \Delta_{F}\}.
$$
In the general case of an admissible edge space $(M,g)$, 
$H_\calE$ is only an initial parametrix, defines a polyhomogeneous function on 
the front face of $\mathscr{M}^2_h$ and solves the heat equation only to first order. 
Repeating almost ad verbatim the heat kernel construction in case of the edge Laplacian in \cite{MazVer:ATM}, we arrive at the following

\begin{prop}\label{heat}
Let $(M^m,g)$ be an incomplete edge space with an admissible edge metric $g$. 
Then the lift $\beta^*H$ is polyhomogeneous on $\mathscr{M}^2_h$ of order 
$(-\dim M)$ at $\ff$ and $\td$, vanishing ot infinite order at $\tf$, and with the
index set at $\rf$ and $\lf$ given by $E+\N_0$ where
$$
E=\{\gamma \geq 0 \mid  \gamma= -\frac{(f-1)}{2}+ \sqrt{\frac{(f-1)^2}{4}+ \sigma^2},
 \, \sigma^2 \in \textup{Spec}\, \Delta_{F,y}\}.
$$
More precisely, if $s$ denotes the boundary defining function of $\rf$, we obtain 
\begin{align*}
\beta^* H \sim \sum_{\gamma \in E} \left(\sum_{j=0}^\infty s^{\gamma+2j} a_{\gamma,j}(\beta^* H) 
+ \sum_{j=0}^\infty s^{\gamma+2+j} a'_{\gamma,j}(\beta^* H)\right)\ \textup{as} \ s\to 0,
\end{align*}
where the coefficients $a_{\gamma,j}(H)$ are of order $(-m)$ at the front face 
and lie in their corresponding $\Delta_{F,y}$ eigenspaces. The higher coefficients $a'_{\gamma,j}(\beta^* H)$
are of order $(-m+1)$ at ff.
\end{prop}

\begin{proof}
Recall the heat kernel construction in \cite{MazVer:ATM}, which we basically follow here.
Denote by $\Delta$ the Laplace Beltrami operator on $(M,g)$.
We write $\mathcal{L}:=\partial_t + \Delta^2$ for the heat operator. 
The restriction of the lift $\beta^*(t\mathcal{L})$ to $\ff$ is called the normal operator 
$N_{\ff}(t\mathcal{L})_{y_0}$ at the front face (at the fibre over $y_0\in B$) 
and is given in projective coordinates \eqref{right-coord} explicitly as follows
\begin{align*}
N_{\ff}(t\mathcal{L})_{y_0}&=\tau \left(\partial_{\tau} + \left(- \partial_s^2 - fs^{-1}\partial_s +
s^{-2}\Delta_{F,y_0} + \Delta^{\R^b}_u\right)^2\right) \\
&=:\tau \left(\partial_{\tau} + \left(\Delta^{\mathscr{C}(F)}_{s,y_0} + \Delta^{\R^b}_u\right)^2\right).
\end{align*}
$N_{\ff}(t\mathcal{L})$ does not involve derivatives with respect to $(y_0,\wx, \wy,\wz)$ and hence acts 
tangentially to the fibres of the front face. Consequently in our choice of an initial parametrix $H_0$ we note that 
the equation
\[
N_{\ff}(t\mathcal{L}) \circ N_{\ff}(H_0) = 0
\]
is the heat equation for the bi-Laplace operator on the model edge $\mathscr{C}(F)\times \R^{b}$. 
Consequently, the initial parametrix $H_0$ is defined by choosing $N_{\ff}(H_0)$ to 
equal the fundamental solution for the heat operator $N_{\ff}(t\calL)$, and extending 
$N_{\ff}(H_0)$ trivially to a neighborhood of the front face. More precisely, 
consider the biharmonic heat kernel $H_{\calE,y}$ on the model edge 
$(\mathscr{C}(F)\times \R^{b}, ds^2 + s^2g^F_{y_0} + du^2$ with the parameter $y_0\in B$. Then in projective coordinates 
$(\tau,s,y_0,z,\wx,u,\wz)$ near the right corner of $\ff$, where $\wx$ is the defining function of the front face, we set
\begin{align}\label{normal-heat-ops}
H_0(\tau, s, u, y_0, z, \widetilde{z}) := 
\wx^{-m} \phi(\wx) H_{\calE,y_0}(\tau, s, u, z, \widetilde{s}=1, \widetilde{u}=0, \widetilde{z}),
\end{align}
where $\phi$ is a smooth cutoff function, $\phi \equiv 1$ in an open neighborhood of 
$\wx=0$, and with compact support in $[0,1)$. By Proposition \ref{model-heat}, 
our initial parametrix $H_0$ is polyhomogeneous on $M^2_h$ and solves the heat equation to 
first order at the front face ff of $\mathscr{M}^2_h$. Moreover Proposition \ref{model-heat} asserts
\begin{align}\label{h-0-exp}
H_0 \sim \sum_{\gamma \in E}\sum_{j=0}^\infty s^{\gamma+2j} a_{\gamma,j}(H_0), \ s\to 0,
\end{align}
with each coefficient $a_{\gamma,j}(H_0)$ lying in the corresponding $\Delta_{F,y_0}$ eigenspace. 
The error of the initial parametrix $H_0$ is given by 
\begin{align*}
\beta^*(t\mathcal{L})H_0=\left(\beta^*(t\Delta^2) - \tau (\Delta^{\mathscr{C}(F)}_{s,y_0} 
+ \Delta^{\R^b}_u)^2\right) H_0=:P_0.
\end{align*}
The leading order term in the expansion of $\beta^*(t\mathcal{L})$ at $\td$ does not 
depend on the edge geometry and corresponds to the bi-Laplacian on a closed manifold.
Consequently, classical arguments allow to refine the initial parametrix such 
that the error term $P_0$ is vanishing to infinite order at $\td$, compare the corresponding 
discussion in (\cite{MazVer:ATM}, Section 3.2).
We need to understand the explicit structure of the asymptotic expansion of $P_0$ at $\ff$ and $\rf$.
By Definition \ref{def-admissible} (iv) we find
\begin{align}\label{delta-terms}
\beta^*\Delta = \wx^{-2}\left(\Delta^{\mathscr{C}(F)}_{s,y_0} + \Delta^{\R^b}_u\right) + \wx^{-1}L_1 + L_2,
\end{align} 
where $L_1$ is comprised of the derivatives $\partial_u\partial_z$ and $L_2$ consists of edge derivatives $\VV$. 
Consequently, we obtain after taking squares
\begin{align*}
\beta^*(t\Delta^2) &- \tau (\Delta^{\mathscr{C}(F)}_{s,y_0} 
+ \Delta^{\R^b}_u)^2 = \tau \wx \left((\Delta^{\mathscr{C}(F)}_{s,y_0} 
+ \Delta^{\R^b}_u)L_1 + L_1 (\Delta^{\mathscr{C}(F)}_{s,y_0} 
+ \Delta^{\R^b}_u)\right) \\
&+ \tau \wx^2 \left((\Delta^{\mathscr{C}(F)}_{s,y_0} 
+ \Delta^{\R^b}_u)L_2 + L_2 (\Delta^{\mathscr{C}(F)}_{s,y_0} 
+ \Delta^{\R^b}_u) + L_1^2\right) + \tau \wx^4 L_2^2.
\end{align*}
We now apply each of the summands above to the asymptotic expansion 
\eqref{h-0-exp} of $H_0$. Note that $\Delta^{\mathscr{C}(F)}_{s,y_0}$ 
annihilates each $s^{\gamma} a_{\gamma,0}(H_0), \gamma \in E$,
and lowers the $s$-order of $s^{\gamma+2j} a_{\gamma,j}(H_0)$ by $2$, if $j\geq 1$.
Consequently, we obtain as $s\to 0$
\begin{align*}
P_0 = \beta^*(t\mathcal{L}) H_0 
\sim \wx^{-m+1}\sum_{\gamma \in E} \sum_{j=0}^{\infty}
s^{\gamma+j-2} c_{\gamma,j}.
\end{align*}
The next step in the construction of the heat kernel involves adding a kernel $H_0'$ to $H_0$, 
such that the new error term is vanishing to infinite order at rf. 
In order to eliminate the term $s^k a$ in the asymptotic 
expansion of $P_0$ at rf, we only need to solve 
\begin{align}\label{indicial}
(-\partial_s^2- f s^{-1}\partial_s + s^{-2}\Delta_{F,y_0})^2 u = s^k (\tau^{-1}a).
\end{align}

This is because all other terms in the expansion of $t\mathcal{L}$ at rf lower the 
exponent in $s$ by at most one, while the indicial part lowers the exponent by two. 
The variables $(\tau, u, \wx, y_0, \wy, \wz)$ enter the equation only as parameters. 
The equation is solved by Mellin transform as well as spectral decomposition over $F$.
The solution is polyhomogeneous in all variables, including parameters and is of leading order $(k+4)$. 
Consequently, the correcting kernel $H_0'$ must be of leading order $2$ at rf 
and of leading order $(-m+1)$ at ff, since $P_0$ is of order $(-2)$ at rf and $(-m+1)$ at ff and
the defining function $\wx$ of the front face enters \eqref{indicial} only as a parameter. Hence
$$
H_1:=H_0+ H_0' \sim \sum_{\gamma \in E} \left(\sum_{j=0}^\infty s^{\gamma+2j} a_{\gamma,j}(H_1) 
+ \sum_{j=0}^\infty s^{\gamma+2+j} a'_{\gamma,j}(H_1)\right)\ \textup{as} \ s\to 0,
$$
where the coefficients $a_{\gamma,j}(H_1)$ each lie in the corresponding $\Delta_{F,y_0}$ eigenspace.
In the following correction steps the exact heat kernel is 
obtained from $H_1$ by an iterative correction procedure, 
adding terms of the form $H_1\circ (P_1)^k$, where $P_1:=t\mathcal{L}H_1$ 
is vanishing to infinite order at rf and td. This leads to an expansion
\begin{align}\label{h-exp}
\beta^* H \sim \sum_{\gamma \in E} \left(\sum_{j=0}^\infty s^{\gamma+2j} a_{\gamma,j}(\beta^* H) 
+ \sum_{j=0}^\infty s^{\gamma+2+j} a'_{\gamma,j}(\beta^* H)\right)\ \textup{as} \ s\to 0,
\end{align}
where the coefficients $a_{\gamma,j}(H)$ still lie in their corresponding $\Delta_{F,y_0}$ eigenspaces, 
and are of order $(-m)$ at the front face. The higher coefficients $a'_{\gamma,j}(\beta^* H)$
are of order $(-m+1)$ at ff.
\medskip

Note that in the construction above, we have only used 
Definition \ref{def-admissible} (i) and (iv). 
The assumption of Definition \ref{def-admissible} (iii) for admissible 
edge metrics is not required there, but plays an essential role in the argument that 
$H$ indeed takes values in $\dom(\Delta^2_{\mathscr{F}})$. First note that $H$ indeed 
takes values in $\dom(\Delta_{\mathscr{F}})$, since the expansion \eqref{h-exp}
satisfies the characterization of maximal solutions in Lemma \ref{max} and also the condition 
in \eqref{F} under the rescaling $\Phi$. \medskip

The conclusion that $\Delta H$ also takes values in $\dom(\Delta_{\mathscr{F}})$ is more intricate.
Recall \eqref{delta-terms}. It is easily checked from \eqref{h-exp} that 
$\wx^{-2}\left(\Delta^{\mathscr{C}(F)}_{s,y_0} + \Delta^{\R^b}_u\right) H$ indeed takes 
values in $\dom(\Delta_{\mathscr{F}})$ without any further assumptions. Application of $(\wx^{-1}L_1 + L_2)$
to $H$ preserves the expansion \eqref{h-exp}, however the coefficients in the expansion need not lie 
in the correct $\Delta_{F,y_0}$ eigenspaces, and hence we cannot deduce that $(\wx^{-1}L_1 + L_2) H$ 
maps into $\dom(\Delta_{\mathscr{F}})$ in general. \medskip

By condition (iii) of Definition \ref{def-admissible}, any $\gamma\neq 0$ is automatically $\gamma > 1$, 
and hence it then suffices to check whether the $s^0$ coefficient in the expansion of $(\wx^{-1}L_1 + L_2) H$ lies 
in the zero-eigenspace of $\Delta_F$, in other words is harmonic on fibres and hence constant in $z$. 
The leading term $s^0 a_{0,0}(\beta^* H)$ in the expansion of $\beta^*H$ is annihilated by 
$(s\partial_s), \partial_z$ and is increased by $\beta^* x\partial_y = s\partial_u + \wx s\partial_y$. 
Consequently, $(\wx^{-1}L_1 + L_2) H$ admits no $s^0$ term and hence trivially maps into $\dom(\Delta_{\mathscr{F}})$. 
\medskip

The kernels $H$ and $\Delta H$ thus both map into $\dom(\Delta_{\mathscr{F}})$
and hence by definition, $H$ indeed maps into $\dom (\Delta^2_{\mathscr{F}})$ and thus is 
the biharmonic heat kernel associated to $\Delta^2_{\mathscr{F}}$. 
\end{proof}

\section{Mapping properties of the biharmonic heat operator}\label{mapping-section}

In this section we prove boundedness and strong continuity of the 
biharmonic heat operator with respect to Banach spaces introduced 
in Definition \ref{defn:norms}.

\begin{thm}\label{bounded}
Let $(M^m,g)$ be an incomplete edge space with an admissible edge metric $g$. 
Fix the Friedrichs extension $\Delta_{\mathscr{F}}$ of the corresponding Laplace 
Beltrami operator. Put $\mathcal{D}_0=\langle\Delta\rangle$ and 
$\mathcal{D}=\langle\Delta,  x^{-1}\VV, x^{-1}\mathcal{V}'_e, \V\rangle$, 
where $\mathcal{V}'_e \subset \V$ consists locally of all edge vector fields where $x\partial_y$
is weighted with functions that are fibrewise constant. Then the associated biharmonic heat operator $e^{-t\Delta_{\mathscr{F}}^2}$
is a bounded map between the (weighted) Banach spaces
$$e^{-t\Delta_{\mathscr{F}}^2}: \mathscr{C}^{2k}_{\textup{ie}}(M, \mathcal{D}_0) 
\to t^{-1/4} \mathscr{C}^{2(k+1)}_{\textup{ie}}(M, \mathcal{D}).$$
\end{thm}

\begin{proof}
First we prove the statement for $k=0$. 
The explicit structure of the heat kernel expansion in Proposition \ref{heat}
implies that for any $X\in \mathcal{D}$ applied to the biharmonic heat kernel $H$,
the lift $\beta^*(XH)$ admits the following behaviour near the front face of the
heat space $\mathscr{M}^2_h$ 
\begin{align}\label{XH}
\beta^*(X H) = O\left((\rho_\rf\rho_\lf)^0 (\rho_\ff \rho_\td)^{-m-2} \rho_\tf^{\infty}\right),
\end{align}
where $\rho_*$ denotes a defining function of a boundary face $*, *\in \{\rf, \lf, \ff, \td, \tf\}$.\medskip

Consider the lift of the volume form in the various projective coordinates near $\ff$.
We explify the transformation rules for the volume form near the lower
left, lower right and the top corner of the front face
\begin{equation}\label{volume1}
\begin{split}
\textup{near left corner:} \quad &\tau=\frac{t}{x^4}=\rho_\tf, \ s=\frac{\wx}{x}=\rho_\lf, \ 
u=\frac{y-\widetilde{y}}{x}, \ x=\rho_\ff, \ y, \ z, \ \widetilde{z}, \\
&\beta^*(\wx^f \, d\wx \, \dvb(\wx))=h \cdot x^{m} \, s^f \, ds\, du\, d\wz, \\
\textup{near right corner:} \quad &\wtau=\frac{t}{\wx^4}=\rho_\tf, \ \ws=\frac{x}{\wx}=\rho_\rf, \ \wu=\frac{y-\widetilde{y}}{\wx}, 
\ z, \ \wx=\rho_\ff, \ \wy, \ \widetilde{z}, \\ &\beta^*(\wx^f \, d\wx \, \dvb(\wx))=
h \cdot \wx^{m} \, \widetilde{\tau}^{-1} \, d\widetilde{\tau}\, d\widetilde{u}\, d\wz, \\
\textup{near top corner:} \quad & \rho=t^{1/4}=\rho_\ff, \  \xi=\frac{x}{\rho}=\rho_\rf, \ 
\widetilde{\xi}=\frac{\widetilde{x}}{\rho}=\rho_\lf, \ u=\frac{y-\widetilde{y}}{\rho}, \ y, \ z, \ \widetilde{z}, \\
&\beta^*(\wx^f \, d\wx \, \dvb(\wx))=h \cdot \rho^{m} \, \widetilde{\xi}^f \, d\widetilde{\xi}\, du\, d\wz.
\end{split}
\end{equation}
The projective coordinates and the transformation rule for the volume form where 
the front and the temporal diagonal faces meet, is as follows
\begin{equation}\label{volume2}
\begin{split}
&\eta=\frac{t^{1/4}}{x}=\rho_\td, \ S=\frac{x-\wx}{t^{1/4}}, \ U=\frac{y-\wy}{t^{1/4}}, \ 
Z=\frac{x(z-\wz)}{t^{1/4}}, \ x=\rho_\ff, \ y, \ z, \\
&\beta^*(\wx^f d\wx \dvb(\wx))=h \cdot x^{m} \eta^m dS\, dU\, dZ.
\end{split}
\end{equation}
When we combine the asymptotics of $\beta^*(X H)$ in \eqref{XH} with the behaviour of the volume form 
in the various projective coordinates \eqref{volume1} and \eqref{volume2}, we find that $\beta^*(XH \wx^f d\wx \dvb(\wx))$
has a singular behaviour of $(\rho_\ff \rho_\td)^{-2} \leq c t^{-1/4}$. Consequently, 
we may estimate for $X\in \mathcal{D}$ and any continuous function, in particular for any $u \in \mathscr{C}_{\textup{ie}}^0(M)$
$$
\|X e^{-t\Delta_{\mathscr{F}}^2} u\|_\infty \leq C t^{-1/4} \|u\|_\infty,
$$ 
for some constant $C>0$ independent of $u$. Furthermore, by Proposition \ref{heat},
$\beta^*XH \sim a_0 \rho_\rf^0$, as $\rho_\rf \to 0$ for $X \in \mathcal{D}$, where $a_0$ is fibrewise 
constant, i.e. independent of $z\in F$. Here the fact that for $X\in x^{-1}\mathcal{V}'_e$ its $\partial_y$
component is weighted with a fibrewise constant function, is essential. Hence, indeed 
$X e^{-t\Delta_{\mathscr{F}}^2} u \in t^{-1/4}\mathscr{C}_{\textup{ie}}^0(M)$. 
This proves the statement for $k=0$. The general statement follows from the fact that
due to \eqref{max-F}, $\ck \subset \dom (\Delta_\mathscr{F}^k)$ and for any $u \in \dom (\Delta_\mathscr{F}^k)$, 
$\Delta^k e^{-t\Delta_{\mathscr{F}}^2} u = e^{-t\Delta_{\mathscr{F}}^2}  \Delta^k  u$, 
by uniqueness of solutions to the biharmonic heat equation with fixed initial condition.\medskip

Finally note that while $t^{1/4}X e^{-t\Delta_{\mathscr{F}}^2} u$
is continuous even for $X \in x^{-1} \V$, it need not remain fibrewise constant at $x=0$ 
since in general $X$  may include vector fields weighted with $z$-dependent functions.
Hence $x^{-1} \V$ is replaced by $x^{1}\mathcal{V}'_e$ in $\mathcal{D}$, where $\mathcal{V}'_e\subset \V$ consists locally 
of linear combinations of $\{x\partial_x, x\partial_y, \partial_z\}$, weighted with functions that are fibrewise constant. 
\end{proof}

\begin{remark}
A particular property\footnote{It suffices that
$u$ is continuously differentiable to first order.} 
of $u \in \mathscr{C}^2_{\textup{ie}}(M, \mathcal{D})$ with 
$\{\partial_x, \partial_y, x^{-1}\partial_z\} \subset \mathcal{D}$ is worth noticing.
In the singular neighborhood of the edge, the distance defined by the Riemannian edge metric $g$ is equivalent to 
$$
d((x,y,z), (\wx,\wy,\wz)) = \left(|x-\wx|^2 + |y-\wy|^2 + (x+\wx)^2|z-\wz|^2\right)^{1/2}.
$$
In local coordinates near the edge we find
\begin{align*}
&u(x,y,z) - u(\wx, \wy, \wz) \\ =& \, u(x, y, z) - u(\wx, y, z) +
u(\wx, y, z) - u(\wx, \wy, z) + u(\wx, \wy, z) - u(\wx, \wy, \wz) \\
=& \, \partial_x u(\xi, y, z) (x-\wx) + \partial_y u(\wx, \gamma, z) (y-\wy) + \wx^{-1}\partial_z u(\wx, \wy, \zeta) \, \wx \, (z-\wz).
\end{align*}
Consequently we obtain
\begin{align*}
|u(x,y,z) - u(\wx, \wy, \wz)|  &\leq ||u||_2 \left(|x-\wx| + |y-\wy| + (x+\wx)|z-\wz|\right) \\
&\leq \sqrt{2} \, ||u||_2 \, d((x,y,z), (\wx,\wy,\wz)).
\end{align*}
In other words, $u \in \mathscr{C}^2_{\textup{ie}}(M, \mathcal{D})$ is automatically Lipschitz with respect to $d$.
\end{remark}

\begin{remark}\label{ie}
We point out that Theorem \ref{bounded} holds also when the basic space $\mathscr{C}^{0}_{\textup{ie}}(M)$
is replaced by the Banach space of sections continuous up to $x=0$, without the requirement of
being fibrewise constant at $\partial M$. Also, we may set $\mathcal{D} = \langle\Delta, x^{-1}\VV, x^{-1}\V, \V\rangle$.
The use of the refined space $\mathscr{C}^{0}_{\textup{ie}}(M)$ and the restriction of $x^{-1}\V$
to $x^{-1}\mathcal{V}'_e$ in $\mathcal{D}$ becomes however crucial in Theorem \ref{strong}.
\end{remark}

\section{Short time existence of semi-linear equations of fourth order}\label{short-section}

In this section we explain how the mapping properties of the biharmonic heat operator and 
its strong continuity yields short-time existence of solutions to certain semilinear equations 
of fourth order. 

\subsection{The Banach fixed point argument}

The underlying idea is based on a fixed point argument in the following theorem.

\begin{theorem} \label{taylor} \cite[Proposition 1.1 in \S 15]{Tay:PDE}
Let $P$ be some, possibly unbounded, linear operator in a Hilbert space $H$, bounded from below.
Suppose that $V,W\subset H$ are Banach spaces, such that $P:V\to W$ is bounded and moreover
\begin{enumerate}
\item $e^{-t P}: V \longrightarrow V$ is a strongly continuous semigroup, for $t \geq 0$.
\item $Q: V \longrightarrow W$ is locally Lipschitz,
\item $e^{-t P}: W \longrightarrow t^{-\gamma}V$ bounded 
for some $\gamma < 1$.
\end{enumerate}
Then for any $u_0 \in W$, the initial value problem 
\begin{equation*}
\partial_t u - P u = Q(u), \ u(0) = u_0 \in W
\end{equation*}
has a unique solution $u \in C([0,T], V)$\footnote{Moreover, by \cite[Proposition 1.2 in \S 15]{Tay:PDE},
$u \in C^\infty ((0,T] \times M)$.}, for some $T > 0$,
where $T$ may be estimated from below in terms of $||u_0||_{V}$.
The solution $u$ is the fixed point of the operator $F:V\to V$ with
\[ F(u) = e^{-t P} u_0 + \int_0^t e^{-(t-s) P} Q(u) ds. \]
\end{theorem}

\subsection{Strong continuity of the biharmonic heat operator}

As seen from Theorem \ref{taylor}, 
existence of solutions to certain semi-linear fourth order equations
crucially depends on the strong continuity property.  Strong continuity of the 
biharmonic heat operator with respect to the Banach space $\ck$ is the 
content of the next theorem. Note that for strong continuity we will choose a different
space $\mathcal{D}'$ of allowable operators, smaller than in Theorem \ref{bounded}.
Beforehand we note the following well-known functional analytic result.

\begin{lemma}\label{funkana}
Let $D$ be a self-adjoint non-negative unbounded operator in a Hilbert space $H$.
Then the following is true.
\begin{enumerate}
\item A solution to $(\partial_t + D^2)u=0$, that is continuously differentiable in $t>0$,
with $u(t)\in \dom(D^2)$ for $t>0$
and $\lim_{t\to 0} u(t) = u_0\in H$, is unique, for any fixed $u_0\in H$.
\item For any $u_0\in \dom (D)$, we have $De^{-tD^2}u_0=e^{-tD^2}Du_0$.
\end{enumerate}
\end{lemma} 

\begin{proof}
(i) For $s\in (0,t]$ we compute
\begin{align*}
\partial_s \, e^{-(t-s)D^2}u(s) &= - \partial_t \, e^{-(t-s)D^2}u(s) + 
e^{-(t-s)D^2} \partial_s \, u(s)
\\&= e^{-(t-s)D^2}D^2u(s)-e^{-(t-s)D^2}D^2u(s)=0.
\end{align*}
Consequently, $e^{-(t-s)D^2}u(s)$ is constant for $s\in (0,t]$. 
Since $e^{-tD^2}$ converges to Id in the Hilbert space norm as $t\to 0$ 
and, by assumption, $u(t)$ is continuous at $t=0$, we find that 
$e^{-(t-s)D^2}u(s)$ is constant for $s\in [0,t]$. 
Considering the limit of $e^{-(t-s)D^2}u(s)$ as $s\to t$ and as $s\to 0$
proves for any $u_0\in H$ $$u(t)=e^{-tD^2}u_0.$$
(ii) Consider $\lambda\in \textup{Res}(D^2)$ in the resolvent set of $D^2$.
Then for any $u_0\in \dom (D)$, the resolvent $(D^2-\lambda)^{-1}u_0\in \dom(D^2)$
and we compute
\begin{align*}
(D^2-\lambda) D (D^2-\lambda)^{-1}u_0 = &D (D^2-\lambda) (D^2-\lambda)^{-1}u_0 =  Du_0, \\
\Rightarrow \, &D (D^2-\lambda)^{-1}u_0 = (D^2-\lambda)^{-1} D u_0.
\end{align*}
The statement now follows by closedness of $D$ and definition 
of the heat operator as the strong limit 
$e^{-tD^2} := \lim_{n\to \infty} \left(I + tD^2/n\right)^{-n}.$
 \end{proof}

\begin{thm}\label{strong}
Let $(M^m,g)$ be an incomplete edge space with an admissible edge metric $g$. 
Fix the Friedrichs extension $\Delta_{\mathscr{F}}$ of the corresponding Laplace-Beltrami operator. 
Put $\mathcal{D}'=\langle\Delta,  \VV, \V \rangle$. 
Then the associated biharmonic heat operator $e^{-t\Delta_{\mathscr{F}}^2}$
is a strongly continuous bounded map between Banach spaces
$$e^{-t\Delta_{\mathscr{F}}^2}: \mathscr{C}^{2k}_{\textup{ie}}(M, \mathcal{D}') \to 
\mathscr{C}^{2k}_{\textup{ie}}(M, \mathcal{D}').$$
\end{thm}

\begin{proof}
By \eqref{max-F}, we have $\mathscr{C}^{2k}_{\textup{ie}}(M, \mathcal{D}')  
\subset \dom (\Delta_\mathscr{F}^k)$ and hence for any 
$u \in \mathscr{C}^{2k}_{\textup{ie}}(M, \mathcal{D}') $ we infer by the previous Lemma \ref{funkana},
$\Delta^k e^{-t\Delta_{\mathscr{F}}^2} u = e^{-t\Delta_{\mathscr{F}}^2}  \Delta^k  u$.
This reduces the statement to $k=1$ and $k=0$. Proof of both cases requires stochastic completeness of the biharmonic 
heat kernel, which we explain below. Solutions to the initial value problem
\begin{align*}
\partial_t u+\Delta^2 u=0, \ u(0)=u_0, \ u(t) \in \dom (\Delta^2_\mathscr{F}), \ t > 0,
\end{align*}
are unique and in fact given by $u(t)=e^{-t\Delta_{\mathscr{F}}^2} u_0\in \dom (\Delta^2_\mathscr{F})$.
We have observed in subsection \ref{friedrichs-subsection} that, reversing eventual rescaling, the Friedrichs domain contains 
precisely those elements in the maximal domain whose leading 
term in the weak expansion as $x\to 0$ is given by $x^0$, with a fibrewise constant coefficient, 
cf. \eqref{max-F}. Consequently, the constant function $\one \in \dom(\Delta_\mathscr{F})$. Moreover, 
$\Delta \one = 0 \in \dom(\Delta_\mathscr{F})$ and consequently $\one\in \dom (\Delta^2_\mathscr{F})$. 
The constant function $\one$ solves the heat equation and hence we deduce by uniqueness 
of solutions the \emph{stochastic completeness}
\begin{align}\label{stoch-compl}
e^{-t\Delta_{\mathscr{F}}^2} \one \equiv \int_M H( t, p,\widetilde{p}) \dv(\widetilde{p}) = \one, \; \mbox{for all} \; p \in M, t > 0. 
\end{align}  

This reduces the case to $k=0,1$. We can now prove the statement for $k=0$, basically repeating the arguments 
in (\cite{BDV:THE}) where the classical proof of strong continuity of the heat operator on closed (non-singular) manifolds 
is adapted to the present setup. Using stochastic completeness we find
\begin{align*}
(e^{-t\Delta_{\mathscr{F}}^2} u)(p, t) - u(p) &= \int_M H\left( t, p, \widetilde{p} \right) (u(\widetilde{p})-u(p)) \dv(\widetilde{p}).
\end{align*}
Consider the distance function $d(p,\widetilde{p})$ induced by the incomplete edge metric $g$. 
In the singular neighborhood of the edge, the distance is equivalent to 
$$
d((x,y,z), (\wx,\wy,\wz)) = \left(|x-\wx|^2 + |y-\wy|^2 + (x+\wx)^2|z-\wz|^2\right)^{1/2}.
$$
Note that $u\in \mathscr{C}^0_{\textup{ie}}(M)$ is continuous with respect to the topology
induced by the Riemannian metric $g$ and hence by the distance function $d$. 
Hence for any $\epsilon >0$ there exists some $\delta(\epsilon) >0$, such that for 
$d(p,\widetilde{p}) \leq \delta(\epsilon)$ one has $|u(p)-u(\widetilde{p})| \leq \epsilon$. 
For any given $\epsilon >0$ we separate the integration region into 
\begin{eqnarray}\label{int_regions}
 M^+_\epsilon & := & \{\widetilde{p} \mid d(p,\widetilde{p}) \geq \delta(\epsilon)\}, \notag \\
 M^-_\epsilon & := & \{\widetilde{p} \mid d(p,\widetilde{p}) \leq \delta(\epsilon)\}.
\end{eqnarray}
Employing continuity of $u$ we find
\begin{equation*}
\begin{split}
 &|e^{-t\Delta_{\mathscr{F}}^2} u - u| = \left| \int_M H\left( t, p, \widetilde{p} \right) (u(\widetilde{p})-u(p)) \dv(\widetilde{p}) \right| \\
&\leq \int_{M^+} |H\left( t, p, \widetilde{p} \right)| \cdot |u(\widetilde{p})-u(p)| \dv(\widetilde{p}) 
+ \int_{M^-} |H\left( t, p, \widetilde{p} \right)| \cdot |u(\widetilde{p})-u(p)| \dv(\widetilde{p}) \\
&\leq  \, 2 \, \frac{t^{1/4}}{\delta(\epsilon)} \, \|u\|_{0}\int_{M^+} |H\left( t, p, \widetilde{p} \right)| \,
\frac{d(p,\widetilde{p})}{t^{1/4}} \dv(\widetilde{p})
+ \epsilon \, \int_{M^-} |H\left( t, p, \widetilde{p} \right)| \dv(\widetilde{p}). 
\end{split}
\end{equation*}
It may be checked in the various projective coordinates around the 
front face in the heat space $\mathscr{M}^2_h$, that $\beta^*(|H| \dv)$ and 
$\beta^*(d(p,\widetilde{p}) t^{-1/4}) \rho_\tf$ is bounded. Since $\beta^*|H|$
is vanishing to infinite order at $\tf$, we find that both integrals above are bounded uniformly in $(t,p, \epsilon)$.
Therefore we obtain
\begin{align*}
 \| e^{-t \Delta } u - u \|_0 \leq C \, \left(\frac{t^{1/4}}{\delta(\epsilon)} \, \|u\|_{0} + \epsilon\right).
\end{align*}
Thus, for any given $\epsilon >0$ we can estimate $\| e^{-t \Delta } u - u \|_0 \leq 2\epsilon C$ for 
$t^{1/4} < \epsilon \delta(\epsilon) / \|u\|_{0}$.
This proves strong continuity of the biharmonic heat operator on $\mathscr{C}^0_{\textup{ie}}(M)$.
It remains to prove the case $k=1$. Strong continuity 
of the biharmonic heat operator with respect to $\mathscr{C}^2_{\textup{ie}}(M)$ means
$\| X(e^{-t\Delta_{\mathscr{F}}^2} u - u)\|_0 \to 0$ as $t\to 0$, for $u \in \mathscr{C}^2_{\textup{ie}}(M)$ and $X\in \mathcal{D}'$. 
If $X= \Delta$, this follows from the case $k=0$, since $\Delta e^{-t\Delta_{\mathscr{F}}^2} u = e^{-t\Delta_{\mathscr{F}}^2}  \Delta u$
for $u \in \mathscr{C}^2_{\textup{ie}}(M)\subset \dom (\Delta_\mathscr{F})$. For $X \in \{\VV, \V\}$ 
the leading order of $\beta^*H$ at the front face is preserved under $X$, so that away from $\td$, the 
estimates reduce to the case $k=0$. \medskip

Near $\td$, a priori $XH$ admits $\rho_\td^{-2}$ singular behaviour at the temporal diagonal.
However, integration by parts, exactly as worked out in detail in \cite{BDV:THE} allows to 
pass derivatives $X$ to $u$, so that
the estimates again reduce to the case $k=0$. We write down the argument for completeness. 
The edge vector fields obey the following transformation rules in projective coordinates \eqref{d-coord} near the 
temporal diagonal 
\begin{align*}
\beta^*(x\partial_x)=-\eta \partial_{\eta} +\frac{1}{\eta}\partial_S + Z \partial_Z + x\partial_x, \ 
\beta^*(x\partial_y)=\frac{1}{\eta}\partial_U + x \partial_y, \ \beta^*(\partial_z)=\frac{1}{\eta}\partial_Z + \partial_z.
\end{align*}
By Proposition \ref{heat}
\begin{align*}
&\beta^*H(\eta, S, U, Z, x, y, z)=x^{-m}\eta^{-m} G(\eta, S, U, Z, x, y, z), \\
&\beta^*(\wx^{f} d\wx \dvb(\wx))=h (x\eta)^m (1-\eta S)^f dS\, dU\, dZ,
\end{align*}
where $G$ is bounded in its entries, and in fact infinitely vanishing as $|(S, U, Z)|\to \infty$, and 
$h = h\left( \eta, x(1-\eta S), y-x\eta U, z-\eta Z, x, y, z \right)$ is a bounded distribution on $\mathscr{M}^2_h$. 
We consider $|| x \partial_x ( e^{-t \Delta} u - u )||_0$. Using stochastic completeness of the heat kernel, we find
\begin{align*}
F&:=x \partial_x ( e^{-t \Delta} u - u)
=\int (x\partial_xH) u(\wx,\wy,\wz)\wx^{f} d\wx \dvb(\wx) \\ &- \int (x\partial_x) [H u(x,y,z) \wx^{f} d\wx \dvb(\wx)]
=:F_1-F_2.
\end{align*}
Next we transform to projective coordinates and integrate by parts in $S$, 
where the boundary terms lie away from the diagonal and hence are vanishing to infinite order for $t\to 0$ 
by the asymptotic behaviour of the heat kernel. Omitting these irrelevant terms,  we obtain
\begin{align*}
F_1&=\int \left(-\eta \partial_{\eta} +\frac{1}{\eta}\partial_S+Z\partial_Z +x\partial_x\right) \left[(x\eta)^{-m}G(\eta, S, U, Z, x, y, z)\right] \\ 
&\times u\left(x(1-\eta S), y-x\eta U, z-\eta Z\right) h (x\eta)^m (1-\eta S)^f dS\, dU\, dZ \\
&= \int \left[(-\eta \partial_{\eta} + Z \partial_Z + x\partial_x) (x\eta)^{-m} G\right]  \cdot u \, h(x\eta)^m (1-\eta S)^f dS\, dU\, dZ \\
&- \int G \left[\left(\frac{1}{\eta}\partial_S\right) u \right] h(1-\eta S)^f dS\, dU\, dZ\\
&- \int (x\eta)^{-m} G \cdot u \left[\left(\frac{1}{\eta}\partial_S \right) h(x\eta)^m(1-\eta S)^f\right] dS\, dU\, dZ.
\end{align*}
We perform similar computations for $F_2$:
\begin{align*}
F_2&=\int \left[(x\partial_x H)u(x,y,z) +H(x\partial_x u)\frac{}{}\right]\wx^f d\wx \dvb(\wx)\\
&=\int \left(\left[-\eta \partial_{\eta} +\frac{1}{\eta}\partial_S+Z\partial_Z +x\partial_x\right] 
(x\eta)^{-m}G\right) u\cdot h(x\eta)^m (1-\eta S)^f dS\, dU\, dZ \\ &+ \int G(\eta, S, U, Z, x, y, z)  
(x\partial_xu(x,y,z)) h (1-\eta S)^f dS\, dU\, dZ \\
&=  \int \left[(-\eta \partial_{\eta} + Z \partial_Z + x\partial_x) (x\eta)^{-m}G\right]  \cdot u\, h(x\eta)^m (1-\eta S)^f dS\, dU\, dZ\\
&- \int (x\eta)^{-m}G \cdot u \left[\left(\frac{1}{\eta}\partial_S \right) h(x\eta)^m(1-\eta S)^f\right] dS\, dU\, dZ\\
&+\int G(\eta, S, U, Z, x, y, z)  (x\partial_xu(x,y,z)) h (1-\eta S)^f dS\, dU\, dZ.
\end{align*}
Thus $F=F_1-F_2$ becomes
\begin{align*}
F&= \int \left[(-\eta \partial_{\eta} + Z \partial_Z + x\partial_x) (x\eta)^{-m}
G(\eta, S, U, Z, x, y, z)\frac{}{}\right] h(x\eta)^m (1-\eta S)^f \\ 
&\times  (u(x(1-\eta S), y-x\eta U, z-\eta Z)-u(x,y,z)) dS\, dU\, dZ \\
&- \int G(\eta, S, U, Z, x, y, z) \left[\left(\frac{1}{\eta}\partial_S\right) h \cdot (1-\eta S)^f\right] \\
&\times  (u(x(1-\eta S), y-x\eta U, z-\eta Z)-u(x,y,z))  dS\, dU\, dZ \\
&- \int G \left[\frac{1}{\eta}\partial_S u(x(1-\eta S), y-x\eta U, z-\eta Z) +
x\partial_x u(x,y,z) \right] h(1-\eta S)^f dS\, dU\, dZ.
\end{align*}
Now, each of the three integrals is estimated as above for $k=0$ by 
separating the integration region into $M^+_\epsilon$ and $M^-_\epsilon$ for 
any given $\epsilon >0$.  Note that in the final integral we use the fact that 
$u \in \mathscr{C}^2_{\textup{ie}}(M, \mathcal{D}')$ so that $\eta^{-1}\partial_Su$ and $x\partial_x u$
are bounded. Higher order and other edge derivatives may be estimated in a similar way. 
This proves strong continuity in general and as a trivial consequence boundedness of
the biharmonic heat operator. 
\end{proof}

\subsection{Existence and regularity of solutions}

We can now establish our final existence and regularity results. 

\begin{cor}
Let $(M,g)$ be an incomplete edge space with an admissible edge metric $g$. 
Put $\mathcal{D}'=\langle\Delta,  \VV, \V \rangle$ and $\mathcal{D}=\langle\Delta,  x^{-1}\VV, x^{-1}\mathcal{V}'_e, \V\rangle$, 
where $\mathcal{V}'_e \subset \V$ consists locally of linear combinations of $\{x\partial_x, x\partial_y, \partial_z\}$, 
where $x\partial_y$ is weighted with functions that are fibrewise constant. Assume $Q: \mathscr{C}^{2(k+1)}_{\textup{ie}}(M, \mathcal{D}')
\to \mathscr{C}^{2k}_{\textup{ie}}(M, \mathcal{D}')$ is locally Lipschitz. Then the semilinear equation
\begin{equation*}
\partial_t u + \Delta^2u = Q(u), \ u(0) = u_0 \in \mathscr{C}^{2(k+1)}_{\textup{ie}}(M, \mathcal{D}')
\end{equation*}
has a unique solution $u \in C([0,T], \mathscr{C}^{2(k+1)}_{\textup{ie}}(M, \mathcal{D})) \cap C^\infty((0,T]\times M)$, for some $T > 0$,
where $T$ may be estimated from below in terms of $||u_0||_{2(k+1)}$.
\end{cor}
\begin{proof}
Consider first a slightly smaller set of operators $\mathcal{D}'=\langle\Delta,  \VV, \V \rangle$ and set 
$W =\mathscr{C}^{2k}_{\textup{ie}}(M, \mathcal{D}'), V=\mathscr{C}^{2(k+1)}_{\textup{ie}}(M, \mathcal{D}')$.
In view of Theorem \ref{bounded} and Theorem \ref{strong}, the heat operator associated to $\Delta_{\mathscr{F}}^2$ 
satisfies the conditions of Theorem \ref{taylor} with $\gamma = 1/4$. Consequently, by 
Theorem \ref{taylor} the unique solution $u$ exists and lies in $\mathscr{C}^{2(k+1)}_{\textup{ie}}(M, \mathcal{D}')$.
This solution is the fixed point of the map $F:\mathscr{C}^{2(k+1)}_{\textup{ie}}(M, \mathcal{D}')\to 
\mathscr{C}^{2(k+1)}_{\textup{ie}}(M, \mathcal{D}')$ with
\[ F(u) = e^{-t \Delta_{\mathscr{F}}^2} u_0 + \int_0^t e^{-(t-s) \Delta_{\mathscr{F}}^2} Q(u) ds. \]
However, by Theorem \ref{bounded}, $F$ actually maps $\mathscr{C}^{2(k+1)}_{\textup{ie}}(M, \mathcal{D}') \subset \mathscr{C}^{2(k+1)}_{\textup{ie}}(M, \mathcal{D}_0)$
to $\mathscr{C}^{2(k+1)}_{\textup{ie}}(M, \mathcal{D})$. Consequently, $u \in \mathscr{C}^{2(k+1)}_{\textup{ie}}(M, \mathcal{D})$, 
with a slightly better regularity, as claimed. 
\end{proof}

We now apply this general existence result to the example of
the Cahn-Hilliard equation on an incomplete edge manifold.
We define
\begin{align*}
Q: \ckk \to \ck, \quad Q(u) := \Delta (u-u^3).
\end{align*}
The mapping $Q$ is indeed locally Lipschitz, since for any $u,v\in \ckk$
\begin{align*}
\|Q(u-v)\|_{2k} &\leq \|\Delta (u-v)\|_{2k} + \|\Delta (u-v)^3\|_{2k}\\
&\leq \|u-v\|_{2(k+1)}\left( 1 + \|u-v\|_{2(k+1)}^2\right).
\end{align*}
We hence arrive at our final result.
\begin{cor}
Let $(M,g)$ be an incomplete edge space with an admissible edge metric $g$. 
Put $\mathcal{D}'=\langle\Delta,  \VV, \V \rangle$ and $\mathcal{D}=\langle\Delta,  x^{-1}\VV, x^{-1}\mathcal{V}'_e, \V\rangle$, 
where $\mathcal{V}'_e \subset \V$ consists locally of linear combinations of $\{x\partial_x, x\partial_y, \partial_z\}$, 
where $x\partial_y$ is weighted with functions that are fibrewise constant. Then the Cahn-Hilliard equation
\begin{equation*}
\partial_t u + \Delta^2u + \Delta (u-u^3)= 0, \ u(0) = u_0 \in \mathscr{C}^{2(k+1)}_{\textup{ie}}(M, \mathcal{D}')
\end{equation*}
has a unique solution $u \in C([0,T], \ckk)\cap C^\infty((0,T]\times M)$, for some $T > 0$.
\end{cor}

\def\cprime{$'$}
\providecommand{\bysame}{\leavevmode\hbox to3em{\hrulefill}\thinspace}
\providecommand{\MR}{\relax\ifhmode\unskip\space\fi MR }
\providecommand{\MRhref}[2]{%
  \href{http://www.ams.org/mathscinet-getitem?mr=#1}{#2}
}
\providecommand{\href}[2]{#2}

\end{document}